\documentclass{amsart}

\usepackage{amssymb,amsmath,latexsym,amsthm}%or any usual package
\usepackage[english]{babel}
\usepackage{tikz}
\usepackage{url}
\usepackage[margin = 1.25in]{geometry}
\usepackage[colorlinks = true, urlcolor = blue, citecolor = blue]{hyperref}

\newtheorem{thm}{Theorem}[section]
\newtheorem{theorem}[thm]{Theorem}
\newtheorem{lemma}[thm]{Lemma}
\newtheorem{corollary}[thm]{Corollary}
\newtheorem{proposition}[thm]{Proposition}

\theoremstyle{definition}
\newtheorem{remark}[thm]{Remark}
\newtheorem{definition}[thm]{Definition}

\numberwithin{equation}{section}

\DeclareMathOperator{\Gal}{Gal}
\DeclareMathOperator{\Disc}{Disc}
\DeclareMathOperator{\NP}{NP}

\newcommand{\set}[1]{\left\lbrace #1 \right\rbrace}
\renewcommand{\a}{\boldsymbol{a}}
\renewcommand{\b}{\boldsymbol{b}}

\begin{document}

\title[Growth of points on hyperelliptic curves over number fields]{Growth of points on hyperelliptic curves over number fields}

%first author

\author{Christopher Keyes}
\address{Christopher Keyes\\
Emory University\\
Department of Mathematics\\
400 Dowman Dr.\\
Atlanta, GA 30322, USA}
\email{christopher.keyes@emory.edu}
\urladdr{\url{http://www.math.emory.edu/~ckeyes3/index.html}}



\thanks{} 

\begin{abstract}
Fix a hyperelliptic curve $C/\mathbb{Q}$ of genus $g$, and consider the number fields $K/\mathbb{Q}$ generated by the algebraic points of $C$. In this paper, we study the number of such extensions with fixed degree $n$ and discriminant bounded by $X$. We show that when $g \geq 1$ and $n$ is sufficiently large relative to the degree of $C$, with $n$ even if $\deg C$ is even, there are $\gg X^{c_n}$ such extensions, where $c_n$ is a positive constant depending on $g$ which tends to $1/4$ as $n \to \infty$. This result builds on work of Lemke Oliver and Thorne who, in the case where $C$ is an elliptic curve, put lower bounds on the number of  extensions with fixed degree and bounded discriminant over which the rank of $C$ grows with specified root number.
\end{abstract}

\maketitle

% French abstract for published version
%\begin{resume}
%Choisissez une courbe hyperelliptique $C/\mathbb{Q}$ de genre $g$, et consid{\'e}rez les corps de nombres alg{\'e}briques $K/\mathbb{Q}$ g{\'e}n{\'e}r{\'e}s par les points alg{\'e}briques de $C$. Dans cet article, nous {\'e}tudions le nombre de ces extensions avec degr{\'e} fixe $n$ et discriminant born{\'e} par $X$. Nous montrons que quand $g \geq 1$ et $n$ est suffisamment grand par rapport au degr{\'e} de $C$, avec $n$ un nombre pair si $\deg C$ est un nombre pair, il existe $\gg X^{c_n}$ de ces extensions, o{\`u} $c_n$ est une constante positive d{\'e}pendant de $g$ qui s'approche de $1/4$ lorsque $n \to \infty$. Ce r{\'e}sultat suit le travail de Lemke Oliver et Thorne qui, dans le cas o{\`u} $C$ est une courbe elliptique, mettent une minoration sur le nombre d'extensions avec le degr{\'e} fixe et discriminant born{\'e} sur lequel le rang grandit avec le num{\'e}ro de racine sp{\'e}cifi{\'e}.
%\end{resume}

%bornes inf{\'e}rieures

\section{Introduction}
\label{sec:intro}

Let $C$ be a smooth projective curve over $\mathbb{Q}$ and fix an algebraic closure $\overline{\mathbb{Q}}$. We say a field $K/\mathbb{Q}$ is \textit{generated by a point of $C$} if $K = \mathbb{Q}(P)$ for some $P \in C(\overline{\mathbb{Q}})$. That is, $K$ is the minimal field of definition for an algebraic point on $C$. For $n \geq 1$ an integer and $X$ a positive real number, we define the quantity $N_{n,C}(X)$ to be the number of such extensions with degree $[K:\mathbb{Q}] = n$ and bounded absolute discriminant $\left|\Disc(K)\right| \leq X$. We also take $N_{n,C}(X, G)$ to be the number of those extensions with $\Gal(\widetilde{K}/\mathbb{Q}) \simeq G$, where $\widetilde{K}$ denotes the Galois closure of $K$.

In their paper on Diophantine Stability, Mazur and Rubin \cite{mazurrubin} ask to what extent the set of fields generated by algebraic points determines the identity of the curve $C$. Motivated by this question, we want to understand how $N_{n,C}(X)$ grows as $X \to \infty$, and how this asymptotic depends on both the geometry of $C$ and the degree $n$. When $C$ is an elliptic curve, Lemke Oliver and Thorne \cite{LOT} show there are $\gg X^{c_n - \epsilon}$ number fields $K/\mathbb{Q}$ of degree $n\geq 2$ and discriminant at most $X$, such that the Mordell--Weil rank of $C(K)$ is greater than that of $C(\mathbb{Q})$, and $C/K$ has specified root number. Here $c_n$ is a positive constant and tends to $1/4$ from below as $n \to \infty$.

In this paper, we consider the case where $C$ is a hyperelliptic curve. Recall a hyperelliptic curve $C/\mathbb{Q}$ is given by an affine equation
\[C \colon y^2 = f(x),\]
where $f(x) \in \mathbb{Q}[x]$. If $f(x)$ is separable then $C$ is nonsingular, and its genus $g$ is related to its degree $d = \deg f$ by
\[ d = \begin{cases} 2g + 1 & d \text{ is odd}\\ 2g+2 & d \text{ is even.} \end{cases}\]

Our main result is an asymptotic lower bound for $N_{n,C}(X, S_n)$ when $n$ is large relative to $d$, which generalizes that of Lemke Oliver and Thorne and recovers their bound when $g=1$. We treat the cases of $d$ odd and even separately in Theorems \ref{thm:odddeg} and \ref{thm:evendeg}. In both cases, the implied constants depend on the degree $n$ of the extension and the model $f$, and we are able to improve our results slightly when $n$ is allowed to be sufficiently large.

\begin{theorem}
\label{thm:odddeg}
	Let $C$ be a hyperelliptic curve with genus $g \geq 1$ and degree $d = 2g + 1$. If $n \geq d$, then
	\[N_{n,C}(X, S_n) \gg X^{c_n}\]
	where 
	\[ c_n = \frac{1}{4} - \frac{gn^2 - (g^2 - 2g - 3)n - 2g^2}{2n^2(n-1)} .\]
	Moreover, if $n$ is sufficiently large, we have the improvement
	\[c_n = \frac{1}{4} - \frac{gn + g^2 - 2g}{2n(n-1)}.\]
\end{theorem}

Theorem \ref{thm:odddeg} applies whenever $C$ has a rational Weierstrass point, as we can choose an equation for $C$ with $d$ odd. In the general case, when $d$ is even, we restrict our attention to even $n$. This turns out to be a necessary restriction because a positive proportion of hyperelliptic curves over $\mathbb{Q}$ will have no points over any odd degree extensions, a result due to Bhargava, Gross, and Wang \cite{BGW}. After making this restriction, we obtain a similar asymptotic bound to Theorem \ref{thm:odddeg}.

\begin{theorem}
\label{thm:evendeg}
	Let $C$ be a hyperelliptic curve with genus $g \geq 1$ and degree $d = 2g + 2$. If $n \geq d+2$ is even, then 
	\[N_{n,C}(X, S_n) \gg X^{c_n}\]
	where
	\[c_n = \frac{1}{4} - \frac{(1 + 2g)n^2 - (2g^2 - 2g - 8)n - (4g^2 + 4g)}{4n^2(n-1)}.\]
	Moreover, when $n$ is sufficiently large, we have the improvement
	\[c_n = \frac{1}{4}  - \frac{(1 + 2g)n - 2g^2 + 2g + 2}{4n(n-1)}.\]
\end{theorem}

\begin{remark}\label{rem:const}
In both cases, the exponent $c_n$ tends to 1/4 from below as $n \to \infty$. If $d > 7$ is odd then $c_n$ is positive for all $n \geq d$. Similarly, if $d \geq 4$ is even then $c_n$ is positive for all $n \geq d+2$. We discuss how to find the threshold where the improved exponent applies and give examples in Section \ref{sec:improve}.
\end{remark}

We contrast Theorems \ref{thm:odddeg} and \ref{thm:evendeg} with a result of Granville \cite{granvilletwists} for quadratic twists of hyperelliptic curves, which tells a very different story for quadratic extensions. Granville proved, assuming the $abc$-conjecture, that when $g \geq 2$, the number of squarefree $d$ such that $|d| \leq D$ and the quadratic twist \[C_d \colon dy^2 = f(x)\] has a nontrivial rational point is $\ll D^{1/(g-1) + o(1)}$. Here, nontrivial refers to points which don't arise from roots of $f(x)$ or points at infinity. Such points on twists give rise to points in $C(\mathbb{Q}(\sqrt{d}))$, suggesting an \textit{upper} bound on $N_{2,C}(X)$ with vanishing exponent as $g \to \infty$.

The proofs of Theorems \ref{thm:odddeg} and \ref{thm:evendeg} employ a similar strategy as that used by Lemke Oliver and Thorne for elliptic curves and large degree fields. The approach is to produce a family of polynomials whose roots give rise to points on $C$. We will contrive this family to consist almost entirely of irreducible polynomials of the desired degree $n$ and Galois group $S_n$. Then we count the family, adjusting for multiplicity of the fields generated, to give a lower bound for $N_{n,C}(X)$.

In Section \ref{sec:group_galois}, we recall the necessary Galois theory to justify using specializations to study irreducibility and Galois groups in polynomial families. We then introduce Newton polygons as a tool to determine Galois groups of polynomials. We apply these results in Section \ref{sec:galois_families} to specific families to show that they are populated by irreducible polynomials with Galois group $S_n$. In Section \ref{sec:intermediates}, we state and prove a useful lemma relating the size of a polynomial’s roots to its coefficients.

These ingredients are assembled in Section \ref{sec:odd} into the proofs of Theorems \ref{thm:odddeg} and \ref{thm:evendeg}. Here, we count specializations of our polynomial families while controlling multiplicity. We show that the contribution by fields with small discriminant is negligible, which improves our final lower bounds slightly. We make further improvements when $n$ is sufficiently large by applying the best known upper bounds on the number of  fixed degree number fields with bounded discriminant due to Lemke Oliver and Thorne \cite{LOT_upper_bounds}.

\subsection*{Acknowledgments}
	The author would like to thank Robert Lemke Oliver for suggesting the topic and for his invaluable guidance, as well as Lea Beneish, Jackson Morrow, and David Zureick-Brown for their helpful comments on previous drafts. The author also thanks the referee for carefully reading the manuscript and making suggestions which improved the paper.
	
\section{Irreducibility and Galois groups in families of polynomials}
\label{sec:group_galois}

\subsection{Hilbert's irreducibility theorem}

A parameterized family of polynomials over $\mathbb{Q}$ is given by a polynomial $f(\boldsymbol{t},x) \in \mathbb{Q}(\boldsymbol{t})[x]$, where $\boldsymbol{t} = (t_1, \ldots t_k)$. If $\boldsymbol{t}_0 \in \mathbb{Q}^k$, then $f_{\boldsymbol{t}_0} = f(\boldsymbol{t}_0,x) \in \mathbb{Q}[x]$ is a \textit{specialization} of $f$. We would like to understand how the irreducibility of $f$ over $\mathbb{Q}(\boldsymbol{t})$ is related to that of its specializations $f_{\boldsymbol{t}_0}$ over $\mathbb{Q}$. Moreover, when $f_{\boldsymbol{t}_0}$ is irreducible, we would like to relate its Galois group $G_{\boldsymbol{t}_0}$ to that of $f$. 

Keeping the notation above, suppose $f$ is irreducible over $\mathbb{Q}(\boldsymbol{t})$. Then the field $K = \mathbb{Q}(\boldsymbol{t})[x]/f(\boldsymbol{t},x)$ is a finite extension of $\mathbb{Q}(\boldsymbol{t})$ of degree $n = \deg f$. The Galois closure of $K/\mathbb{Q}(\boldsymbol{t})$ is denoted $\widetilde{K}$, allowing us to define the Galois group $G = \Gal(f/\mathbb{Q}(\boldsymbol{t})) = \Gal(\widetilde{K}/\mathbb{Q}(\boldsymbol{t}))$. Let $g(\boldsymbol{t}, x) \in \mathbb{Q}(\boldsymbol{t},x)$ generate the extension $\widetilde{K}$, that is $\widetilde{K} \simeq \mathbb{Q}(\boldsymbol{t})[x]/g(\boldsymbol{t},x)$. Again, we use $g_{\boldsymbol{t}_0}$ to denote the specialization $g(\boldsymbol{t}_0,x) \in \mathbb{Q}[x]$.

\begin{theorem}[Hilbert irreducibility]
\label{thm:hilbert_irred}
	Using the notation above, suppose $\boldsymbol{t}_0 \in \mathbb{Q}^k$ is such that $g_{\boldsymbol{t}_0}$ is irreducible over $\mathbb{Q}$. Then the permutation representations of $G$ and $G_{\boldsymbol{t}_0}$ acting on the roots of $f$ and $f_{\boldsymbol{t}_0}$ are isomorphic.
	
	Moreover, the proportion of specializations $g_{\boldsymbol{t}_0}$ which are irreducible is $1-o_H(1)$ for $\boldsymbol{t}_0$ in any rectangular region in $\mathbb{Z}^k$ having shortest side length $H$.
\end{theorem}

The fact that almost all specializations of an irreducible polynomial are irreducible is classical. For a proof of the isomorphism of permutation representations, we refer the reader to \cite[Theorem 4.1]{LOT}. Theorem \ref{thm:hilbert_irred} tells us that once we know $f$, and hence $g$, is irreducible over $\mathbb{Q}(\boldsymbol{t})$, then 100\% of its integral specializations $f_{\boldsymbol{t}_0}$ are irreducible over $\mathbb{Q}$ with Galois groups $G \simeq G_{\boldsymbol{t}_0}$ isomorphic as permutation groups.

From Theorem \ref{thm:hilbert_irred} we obtain the following corollary, also appearing in \cite[Corollary 4.2]{LOT}. 

\begin{corollary}
\label{cor:hilbert_cor}
	Suppose $f(\boldsymbol{t}, x)$ is irreducible over $\mathbb{Q}(\boldsymbol{t})$. If the permutation representation of $G_{\boldsymbol{t}_0}$ contains an element of a given cycle type for a positive proportion of $\boldsymbol{t}_0 \in \mathbb{Z}^k$, then the permutation representation of $G$ must also contain an element of that type.
\end{corollary}

\begin{proof}
	By Theorem \ref{thm:hilbert_irred}, choose a sufficiently large rectangular region in $\mathbb{Z}^k$ such that the proportion of $\boldsymbol{t}_0$ for which $G$ and $G_{\boldsymbol{t}_0}$ do not have isomorphic permutation representations is smaller than the proportion of $\boldsymbol{t}_0$ for which $G_{\boldsymbol{t}_0}$ contains an element of the given cycle type. Then there is some $\boldsymbol{t}_0$ in the region such that both $G \simeq G_{\boldsymbol{t}_0}$ and $G_{\boldsymbol{t}_0}$ contains the given cycle type, thus $G$ must contain that cycle type.
\end{proof}	

The upshot of Corollary \ref{cor:hilbert_cor} is that we need only prove that a positive proportion of integral specializations $f_{\boldsymbol{t}_0}$ have Galois group $G_{\boldsymbol{t}_0}$ containing a given cycle type to see that $G$ does. Then by another application of Theorem \ref{thm:hilbert_irred}, 100\% of specializations contain an element with the given type. In particular, if the presence of certain cycle types in $G$ implies that $G$ is the full symmetric group $S_n$, then we need only find that positive proportions of specializations $f_{\boldsymbol{t}_0}$ have each of these types to see that 100\% of specializations have Galois group $S_n$.

\subsection{Recognizing the symmetric group}
Let $S_n$ denote the symmetric group acting on the set $\left\{1, \ldots, n\right\}$, and let $G \subseteq S_n$ be a permutation subgroup. Recall that $G$ is a \textit{transitive} subgroup if for all pairs $i,j$ with $1 \leq i,j \leq n$, there exists an element $\sigma \in G$ for which $\sigma(i) = j$. We will describe several ways to detect that a transitive subgroup $G$ is isomorphic to the full symmetric group $S_n$ using the presence of certain cycle types.

\begin{lemma}[Lifting transitive subgroups]
\label{lem:transitive_lifting}
	Let $G \subseteq S_n$ be a transitive permutation subgroup on the set $\set{1, \ldots, n}$. Assume $G$ contains a subgroup $H$ which is isomorphic as a permutation subgroup to $S_k$ for some $k > n/2$. Then $G \simeq S_n$.
\end{lemma}

\begin{proof}
	The result is clearly true if $k = n$, so assume $k < n$. After renumbering if necessary, we may assume $H$ acts nontrivially on $\set{1, \ldots, k}$ and acts trivially on $\set{k+1, \ldots, n}$. In particular, $G$ contains the transpositions $(1 \ a)$ for $2 \leq a \leq k$. Let $\sigma_j \in G$ be an element such that $\sigma_j(1) = j$, which exists by the transitivity of $G$. Then $\sigma_j(a)$ takes on $k-1$ different values for $2 \leq a \leq k$, none of which are equal to $j$. 

	Set $j = k+1$, so we have that $\sigma_{k+1}(a)$ takes $k-1$ distinct values when $2 \leq a \leq k$, none of which are equal to $k+1$. The hypothesis that $k > n/2$ is equivalent to $k-1 > n-k-1$, so by the pigeonhole principle, there exists at least one such $a$ for which $\sigma_{k+1}(a) \leq k$. Conjugating $(1 \ a)$ by $\sigma_{k+1}$, we see that $G$ contains the transposition $(\sigma_{k+1}(a) \ k+1)$. Together with the subgroup $H$, this transposition generates a subgroup of $G$ isomorphic to $S_{k+1}$. Finitely many applications of this procedure show $G \simeq S_n$.
\end{proof}

\begin{proposition}
\label{prop:gen_sets}
	Suppose $G \subseteq S_n$ is a transitive permutation subgroup containing a transposition, $\tau$, and a cycle, $\sigma$, of length $n-1$ or length $p > n/2$ for $p$ a prime. Then $G \simeq S_n$.
\end{proposition}

\begin{proof}
	Suppose first that $\sigma$ has length $n-1$ and renumber so that $\sigma$ is given by $(1 \ \ldots \ n-1)$ in cycle notation. Write $\tau = (a \ b)$. Since $G$ is transitive, we can conjugate $\tau$ by some element of $G$ to produce a transposition $(n \ c)$ where $1 \leq c \leq n-1$. Conjugation of $(n \ c)$ by powers of $\sigma$ produces $\set{(n \ d) \mid 1 \leq d \leq n-1} \subseteq G$, which is a generating set for $S_n$. Hence $G \simeq S_n$.
	
	Now suppose that $\sigma$ has length $p$ for some prime $p > n/2$, and again renumber so $\sigma = (1 \ \ldots \ p)$. Conjugating $\tau$ produces a transposition $(1 \ b) \in G$ for some $b$. Suppose that $b > p$. Then conjugation of $(1 \ b)$ by powers of $\sigma$ produces the subset $\set{(a \ b) \mid 1 \leq a \leq p} \subseteq G$, which generates a subgroup $H \subseteq G$ isomorphic to $S_{p+1}$ acting on $\set{1, \ldots, p, b}$. 
	
	If instead $1 < b \leq p$, then $\sigma^i(1) = (1 \ b \ \ldots)$ is a $p$-cycle for some $1 \leq i < p$. We may renumber again such that $b=2$, making our transposition $(1 \ 2)$, and $\sigma^i = (1 \ 2 \ \ldots \ p)$. Conjugating $(1 \ 2)$ by powers of $\sigma^i$, we obtain $\set{(c \ c+1) \mid 1 \leq c \leq p-1} \subseteq G$, which is a generating set for $S_p$ acting on $\set{1, \ldots, p}$.
		
	In either case, we have shown that there exists a permutation subgroup $H \subseteq G$ such that $H \simeq S_k$ for $k = p$ or $k = p+1$. Since $p > n/2$, the hypothesis of Lemma \ref{lem:transitive_lifting} applies, so we may conclude $G \simeq S_n$.
\end{proof}

\subsection{Newton polygons}

We now introduce the Newton polygon, our tool for showing that a positive proportion of integral specializations $f_{\boldsymbol{t}_0}$ have certain cycle types in their Galois group. Let $p$ be a prime, $\mathbb{Q}_p$ the field of $p$-adic numbers, and $f(x) \in \mathbb{Q}_p[x]$ a polynomial.

\begin{definition}[Newton polygon]
\label{def:newton_polygon}
	With the notation above, let $f(x)$ be given by $f(x) = \sum_{i=0}^n k_ix^i$. The \textbf{Newton polygon of $f$} is the lower convex hull of the set
	\[\set{(i, v_p(k_i)) \in \mathbb{R}^2 \mid 0 \leq i \leq n},\]
	where $v_p$ denotes the $p$-adic valuation, and we set $v_p(0) = \infty$ by convention. We will denote the Newton polygon of $f$ by $\NP_{\mathbb{Q}_p}(f)$, and simply $\NP(f)$ when it will not create confusion.
\end{definition}

The Newton polygon $\NP(f)$ can be split up into segments of distinct increasing slopes. The number and slopes of segments in the Newton polygon determine the valuations of the roots of $f(x)$ in $\mathbb{Q}_p$. More precisely, if a segment of the Newton polygon with length $l$ has slope $s$ then $f(x)$ has $l$ roots each with valuation $-s$ in $\overline{\mathbb{Q}_p}$. For a proof, see \cite[II.6]{Neukirch}. This key fact allows us to prove two lemmas.

\begin{lemma}
\label{lem:NP_factorization}
	Suppose $\NP(f)$ has a segment of length $l$ and slope $s$, and no other segments of this slope (i.e. consider the \textit{entire} segment of slope $s$). Then $f$ has a factorization $f = f_0f_1$ over $\mathbb{Q}_p$, such that $\deg f_0 = l$ and the roots of $f_0$ have valuation $-s$.
	
	Moreover, if $s = r/l$ has reduced fraction form $r'/l'$ then all irreducible factors of $f_0$ over $\mathbb{Q}_p$ have degree divisible by $l'$. In particular, if $\gcd(r,l) = 1$ then the $f_0$ produced above is irreducible over $\mathbb{Q}_p$.
\end{lemma}

\begin{proof}
	Since the action of the Galois group $\Gal(f/\mathbb{Q}_p)$ on the roots of $f$ preserves their valuations, we see that for an irreducible polynomial over $\mathbb{Q}_p$, all roots have the same valuation. Therefore, we can decompose $f$ into irreducible factors and group together those whose roots have valuation $-s$ into $f_0$. This must have degree $l$, since $f$ has exactly $l$ roots with valuation $-s$.
	
	For the second statement, we use the same observation above to recognize that the Newton polygon of an irreducible polynomial has one segment. Let $g$ be an irreducible polynomial over $\mathbb{Q}_p$ dividing $f_0$. Then $\NP(g)$ has one segment of slope $s = r_g / \deg g$. Since reducing this fraction also produces $r'/l'$, we must have $l' \mid \deg G$.
\end{proof}

\begin{lemma}
\label{lem:NP_cycle}
	Suppose $f(x) \in \mathbb{Q}[x]$, $p > \deg f$, and $\NP_{\mathbb{Q}_p}(f)$ has a segment of length $l$ and slope $r/l$ such that $\gcd(r,l) = 1$. If $f = f_0f_1$ is the factorization from Lemma \ref{lem:NP_factorization} and $l$ is pairwise coprime to the degrees of the irreducible factors of $f_1$ over $\mathbb{Q}_p$, then $\Gal(f/\mathbb{Q})$ contains an $l$-cycle.
\end{lemma}

\begin{proof}
		We begin by factoring $f = f_0f_1$ as in Lemma \ref{lem:NP_factorization}, noting that $f_0$ must be irreducible of degree $l$. Let $E_0, E_1, E = E_0E_1$ be the splitting fields of $f_0, f_1, f$ respectively, obtained by adjoining roots. Let $T_0, T_1, T$ denote the maximal unramified subextensions of $E_0, E_1, E$ over $\mathbb{Q}_p$. We will find an $l$-cycle in $\Gal(E_0/T_0)$ and make identifications
	\begin{equation}
	\label{eq:galois_identifications}
		\Gal(E_0/T_0) \simeq \Gal(E_0T_0T_1/T_0T_1) \simeq \Gal(E/E_1T_0T_1) \subseteq \Gal(E/\mathbb{Q}_p).
	\end{equation}
	Identifying $\Gal(E/\mathbb{Q}_p)$ with $\Gal(f/\mathbb{Q}_p)$ and taking the natural inclusion into $\Gal(f/\mathbb{Q})$ gives the result.
	
	Let $L = \mathbb{Q}_p[x] / (f_0(x))$, which is a degree $l$ extension of $\mathbb{Q}_p$. By hypothesis, the set of $p$-adic valuations $v_p(L^\times)$ contains $\frac{r}{l}$, and $\gcd(r,l) = 1$ implies that $\frac{1}{l}\mathbb{Z} \subseteq v_p(L^\times)$. Hence, the ramification index of $L/\mathbb{Q}_p$ is equal to $l$, making $L/\mathbb{Q}_p$ totally ramified. 
	
	The hypothesis $p > \deg f$ implies $p \nmid l$, meaning $L/\mathbb{Q}_p$ is totally tamely ramified, so there exists a uniformizer $\pi \in L$ which satisfies $x^l - p = 0$ \cite[II.5, Proposition 12]{lang}, and thus $L = \mathbb{Q}_p(\pi)$. Whether or not $L/\mathbb{Q}_p$ is Galois, the Galois closure $E_0$ has an cyclic automorphism of order $l$, coming from a primitive $l$-th root of unity. This automorphism fixes $T_0$ and necessarily acts nontrivially on $L$, permuting the roots of $f_0$ in a cyclic fashion.
	
	The first identification in \eqref{eq:galois_identifications} follows from elementary Galois theory, since $E_0 \cap T_0T_1 = T_0$ and $E_0/T$ is a Galois extension. See e.g.\ \cite[\S 14.4, Proposition 19]{dummitfoote}.
	
	For the second identification in \eqref{eq:galois_identifications}, we remark that the ramification index $[E_1 : T_1]$ divides the product of the degrees of the irreducible factors of $f_1$ over $\mathbb{Q}_p$, which is coprime to $l$ by our hypotheses. Moreover, we have $[E_1: T_1] = [E_1T_0T_1: T_0T_1]$ since $E_1 \cap T_0T_1 = T_1$. Thus $l = [E_0: T_0] = [E_0T_0T_1:T_0T_1]$ is coprime to $[E_1T_0T_1: T_0T_1]$, so $E_0T_0T_1 \cap E_1T_0T_1 = T_0T_1$, and the identification of Galois groups follows from the same argument as the first one.
\end{proof}

To connect this result to families of polynomials, suppose we have an irreducible $f(\boldsymbol{t},x) \in \mathbb{Q}(\boldsymbol{t})[x]$, which is equivalent to its Galois group $G$ being a transitive subgroup of $S_n$. We look for integral specializations $\boldsymbol{t}_0 \in \mathbb{Z}^k$ such that for some prime $p$, the Newton polygon of $f_{\boldsymbol{t}_0}$ has a segment satisfying the hypothesis of Lemma \ref{lem:NP_cycle}, and hence $G_{\boldsymbol{t}_0}$ contains an $l$-cycle. If we find that we need only specify the $p$-adic valuations of the specialization $\boldsymbol{t}_0$ to obtain such a cycle, then $G_{\boldsymbol{t}_0}$ contains an $l$-cycle for a positive proportion of $\boldsymbol{t}_0 \in \mathbb{Z}^k$, and Corollary \ref{cor:hilbert_cor} implies that $G$ contains an $l$-cycle as well. Repeating this procedure to find cycles of different lengths, we can hope to satisfy the hypotheses of Proposition \ref{prop:gen_sets} to see that $G$ is in fact the full symmetric group, in which case Theorem \ref{thm:hilbert_irred} implies 100\% of specializations have $G_{\boldsymbol{t}_0} = S_n$. We realize this procedure in the following section for specific polynomial families.

\section{Polynomial families arising from hyperelliptic curves}
\label{sec:galois_families}

Suppose $C/\mathbb{Q}$ is a nonsingular hyperelliptic curve given by
\[C \colon y^2 = f(x) = \sum_{i=0}^d c_ix^i\]
for a squarefree polynomial $f(x) \in \mathbb{Z}[x]$. In this section, we construct families of polynomials whose specializations give rise to number fields generated by algebraic points on $C$.

Let $g(x) = \sum_{i=0}^{d_g} a_ix^i \in \mathbb{Q}(\a)[x]$ and $h(x) = \sum_{i=0}^{d_h} b_ix^i \in \mathbb{Q}(\b)[x]$, where $\a = (a_0, \ldots a_{d_g})$ and $\b = (b_0, \ldots, b_{d_h})$. Then consider the polynomial in $\mathbb{Q}(\a,\b)[x]$ given by
\begin{equation}\label{eq:F_f}
F_f(\a,\b,x) = g(x)^2 - f(x)h(x)^2,
\end{equation}
which has degree $n = \max(2d_g, d + 2d_h)$. We will use $F_{f,\a_0, \b_0}(x)$ to denote a specialization with $\a_0 \in \mathbb{Q}^{d_g+1}$ and $\b_0 \in \mathbb{Q}^{d_h+1}$.

Let $\a_0, \b_0$ be rational specializations such that $F_{f,\a_0,\b_0}(x)$ has degree $n$ over $\mathbb{Q}$, and suppose further that it is irreducible. Take $\alpha$ to be a root of $F_{f,\a_0,\b_0}(x)$, which by rearranging \eqref{eq:F_f} satisfies
\[\left(\frac{g_{\a_0}(\alpha)}{h_{\b_0}(\alpha)}\right)^2 = f(\alpha).\]
Thus we have $P = \left(\alpha, \frac{g_{\a_0}(\alpha)}{h_{\b_0}(\alpha)}\right)$ is an algebraic point on $C$ and $\mathbb{Q}(P)$ is precisely the degree $n$ field $\mathbb{Q}(\alpha)$.

Given $f(x)$ and a degree $n$, our goal is now to describe a polynomial family $F_f(\a,\b,x)$, and use the methods of the previous section to prove that it is irreducible over $\mathbb{Q}(\a,\b)$ with Galois group $G \simeq S_n$. This will give us a means of producing many degree $n$ number fields which are generated by algebraic points of $C$, which we can count later.

\subsection{Odd degree curves} Fix $f$ with odd degree $d \geq 3$. Fix a degree $n \geq d$. We take the degrees $d_g$ and $d_h$ to be as large as possible so \eqref{eq:F_f} has degree $n$,
\begin{align*}
	d_g &= \begin{cases} (n-1)/2, & n \text{ odd,}\\ n/2, & n \text{ even,} \end{cases}\\
	d_h &= \begin{cases} (n-d)/2, & n \text{ odd,}\\ (n-d-1)/2, & n \text{ even.} \end{cases}\end{align*}
For simplicity, we denote the polynomial family \eqref{eq:F_f} by $F(x) \in \mathbb{Q}(\a,\b)[x]$ and a specialization by $F_{\a_0,\b_0}(x) \in \mathbb{Q}[x]$, leaving both $f$ and $n$ implicit when it will not create confusion.

\begin{proposition}
\label{prop:odd_irred_Sn}
	Fix a polynomial $f$ and integers $n, d_g, d_h$ as above. Then $F_f$ is irreducible in $\mathbb{Q}(\a,\b)[x]$ and $\Gal(F_f/\mathbb{Q}(\a,\b)) \simeq S_n$.
\end{proposition}

\begin{proof}

The irreducibility and Galois group of $F_f(x)$ over $\mathbb{Q}(\a,\b)$ are invariant under a linear change of variables in $x$. It will be convenient to assume that the constant term of $f$, $c_0$, is nonzero, which is always possible after a linear change of variables. We treat the cases of $n$ even and odd separately.\bigskip

\textit{Case 1: $n$ is even.} When $n$ is even, we take $d_g = n/2$ and $d_h = (n-d-1)/2$. Let $p$ be a prime that does not divide any nonzero coefficient of $f$. Consider an integral specialization $\a_0 = (a_0, \ldots, a_{n/2})$ and $\b_0 = (b_0, \ldots, b_{(n-d-1)/2})$ with the following $p$-adic valuations:
\begin{align}
\label{eq:odd_d_even_n_split}
	v_{p}(a_0) &= 1 \\
	\nonumber v_{p}(a_i) & \geq 1 \text{ for } 0 < i < n/2\\
	\nonumber v_{p}(a_{n/2}) &= 0\\
	\nonumber v_{p}(b_j) &\geq 2 \text{ for } 0 \leq j \leq (n-d-1)/2.
\end{align}
These requirements on the valuations of $b_j$ allow us to effectively ignore the $h_{\b_0}(x)^2f(x)$ term of $F_{\a_0, \b_0}$ in constructing the Newton polygon. Inspecting the valuations of the coefficients of $g_{\a_0}(x)^2$ gives the resulting $\mathbb{Q}_{p}$-adic Newton polygon for $F_{\a_0, \b_0}$, shown in Figure \ref{fig:odd_d_even_n_split_NP}. 

\begin{figure}[ht]
\centering
\caption{$\NP_{\mathbb{Q}_{p}}(F_{\a_0,\b_0})$ with one segment of slope $-2/n$}
\label{fig:odd_d_even_n_split_NP}
\vspace{1ex}

\begin{tikzpicture}[scale=0.75]
	\draw[->, thick] (-0.2, 0) -- (10.5, 0);
	\draw[->, thick] (0, -0.2) -- (0, 2.5);
	
	\filldraw[black] (0, 2) circle (2pt) node[left] {$(0, 2)$};
	\filldraw[black] (5,1) circle (2pt) node[above] {$(n/2, 1)$};
	\filldraw[black] (10, 0) circle (2pt) node[above] {$(n, 0)$};
	
	\draw[black] (0,2) -- (10, 0);
	
\end{tikzpicture}
\end{figure}

The Newton polygon $\NP_{\mathbb{Q}_{p}}(F_{\a_0,\b_0})$ has one segment of slope $-2/n$, so by Lemma \ref{lem:NP_factorization}, if $F_{\a_0,\b_0}$ is reducible over $\mathbb{Q}_{p}$ then it is the product of two degree $n/2$ irreducible factors. In particular, if $F$ is reducible over $\mathbb{Q}(\a,\b)$, it must also be the product of two degree $n/2$ irreducible factors, as any other factorization would yield an incompatible factorization upon specializing by $\a_0, \b_0$ with the valuations given in \eqref{eq:odd_d_even_n_split}.

Let us now consider a different integral specialization $\a_0, \b_0$ with the following $p$-adic valuations:
\begin{align}
\label{eq:odd_d_even_n_n-1_cycle}
	v_p(a_0) &= 0 \\
	\nonumber v_p(a_i) & \geq 2 \text{ for } 0 < i \leq n/2\\
	\nonumber v_p(b_j) & \geq 2 \text{ for } 0 \leq j < (n-d-1)/2\\
	\nonumber v_p(b_{(n-d-1)/2}) & = 1.
\end{align}
The constant term of $F_{\a_0,\b_0}$ is $a_0^2 - b_0^2c_0$ which has valuation 0. All other coefficients can be seen to have valuation at least 2, with the leading coefficient having valuation at least 4. The coefficient of $x^{n-1}$ is given by $2a_{n/2-1}a_{n/2} - b_{(n-d-1)/2}^2c_d$, which has valuation exactly 2. The resulting Newton polygon is shown below in Figure \ref{fig:odd_d_even_n_n-1_cycle}. 

\begin{figure}[ht]
\centering
\caption{$\NP_{\mathbb{Q}_p}(F_{\a_0,\b_0})$ with $(n-1)$-cycle}
\label{fig:odd_d_even_n_n-1_cycle}
\vspace{1ex}

\begin{tikzpicture}[scale=0.75]
	\draw[->, thick] (-0.2, 0) -- (10.2, 0);
	\draw[->, thick] (0, -0.2) -- (0, 3.2);
	
	\filldraw[black] (0, 0) circle (2pt) node[left] {$(0, 0)$};
%	\filldraw[black] (5, 0) circle (2pt) node[below] {$(n/2, 0)$};
%	\filldraw[black] (7.5, 0) circle (2pt);
%	\filldraw[black] (8, 0) circle (2pt);
%	\filldraw[black] (8.5, 0) circle (2pt);
%	\filldraw[black] (9, 0) circle (2pt);
	\filldraw[black] (9.5, 1.5) circle (2pt) node[right] {$(n-1, 2)$};
	\filldraw[black] (10, 3) circle (2pt) node[right] {$(n, \geq 4)$};
	
	\draw[black] (0,0) -- (9.5, 1.5);
	\draw[black] (9.5, 1.5) -- (10, 3);
	
%	\draw (5, 3) node {i) $v_{p_1}(a_{n/2}) = 1, v_{p_1}(a_0) = -1$};
\end{tikzpicture}
\end{figure}

This Newton polygon has a segment of length $n-1$ and slope equal to $2/(n-1)$, so by Lemma \ref{lem:NP_factorization} whenever $\a_0, \b_0$ have the $p$-adic valuations given in \eqref{eq:odd_d_even_n_n-1_cycle}, we have that $F_{\a_0,\b_0}$ factors as a degree $n-1$ irreducible polynomial times a linear polynomial over $\mathbb{Q}_p$. Such a factorization cannot occur if $F$ has two irreducible degree $n/2$ factors over $\mathbb{Q}(\a,\b)$, so we may conclude that $F$ is irreducible, and hence $G$ is a transitive permutation subgroup of $S_n$. Moreover, Lemma \ref{lem:NP_cycle} implies that the Galois group of $F_{\a_0,\b_0}$ over $\mathbb{Q}$ contains a cycle of length $n-1$ whenever $\a_0$ and $\b_0$ satisfy the valuations in \eqref{eq:odd_d_even_n_n-1_cycle}. These valuation criteria are satisfied for a positive proportion of integral specializations $\a_0$ and $\b_0$, so Corollary \ref{cor:hilbert_cor} implies that $G$ contains an ($n-1$)-cycle.

To produce a transposition in $G$, we first argue that the set of primes $p$ such that $f(k) \equiv 0 \pmod{p}$ for some $k \in \mathbb{Z}$, is infinite. This fact has an elementary proof, but is also seen to be a consequence of the Chebotarev density theorem (see e.g.\ \cite[\S 13]{Neukirch}), as this set contains the positive density set of primes which split completely in the splitting field of $f(x)$. In any case, we may fix a prime $p > n$ with $p \nmid \Disc f, c_d$ such that $p \mid f(k)$ for some integer $k$, which implies that $p$ divides the constant coefficient of the translation $f(x+k)$, but not the linear term, as $p \nmid \Disc f$ implies that the reduction of $f(x)$ mod $p$ is also squarefree. Using a Hensel's lemma lifting argument, we can further find an integral solution to $f(x) \equiv 0 \pmod{p}$ such that $f(x) \not\equiv 0 \pmod{p^2}$. Thus after possibly another change of variables, we may assume that $v_p(c_0) = 1$ and $v_p(c_1) = 0$.

We consider an integral specialization $\a_0, \b_0$ with the following $p$-adic valuations:
\begin{align}
\label{eq:odd_d_even_n_transp}
	v_p(a_0) &= 2 \\
\nonumber	v_p(a_1) &= 0 \\
\nonumber	v_p(a_i) &\geq 2 \text{ for } 1 < i < n/2\\
\nonumber v_p(a_{n/2}) &= 3\\
\nonumber	v_p(b_0) = v_p(b_{(n-d-1)/2}) &= 1 \\
\nonumber	v_p(b_j) &\geq 1 \text{ for } 0 < j < (n-d-1)/2.
\end{align}
These requirements ensure that the constant term of $F_{\a_0, \b_0}$ has valuation exactly 3, the coefficient of $x^2$ has valuation exactly 0, the $x^{n-1}$ coefficient $2a_{n/2}a_{n/2-1} - b_{(n-d-1)/2}^2c_d$ has valuation exactly 2, and the leading term has valuation exactly 6, with all other coefficients having valuation at least 2. The resulting Newton polygon is shown below in Figure \ref{fig:odd_d_even_n_transp}.

\begin{figure}[ht]
\centering
\caption{$\NP_{\mathbb{Q}_p}(F_{\a_0,\b_0})$ with transposition}
\label{fig:odd_d_even_n_transp}
\vspace{1ex}

\begin{tikzpicture}[scale=0.75]
	\draw[->, thick] (-0.2, 0) -- (10.5, 0);
	\draw[->, thick] (0, -0.2) -- (0, 3.5);
	
	\filldraw[black] (0, 1.5) circle (2pt) node[left] {$(0, 3)$};
	\filldraw[black] (2, 0) circle (2pt) node[below] {$(2, 0)$};
	\filldraw[black] (9, 1) circle (2pt) node[right] {$(n-1, 2)$};
	\filldraw[black] (10, 3) circle (2pt) node[left] {$(n, 6)$};
	
	\draw[black] (0,1.5) -- (2,0);
	\draw[black] (2,0) -- (9,1);
	\draw[black] (10,3) -- (9,1);
\end{tikzpicture}
\end{figure}

That $n > d \geq 3$ ensures that $\frac{2}{n-3} < 4$, so the two rightmost segments are distinct. These, together with the segment of length 2 and slope $-3/2$ above, ensure that $F_{\a_0, \b_0}$ has factors of degree 2, $n-3$, and 1 over $\mathbb{Q}_p$, so Lemma \ref{lem:NP_cycle} applies to reveal a transposition in $G_{\a_0, \b_0}$.

Since a positive proportion of integer tuples $\a_0, \b_0$ satisfy \eqref{eq:odd_d_even_n_transp}, Corollary \ref{cor:hilbert_cor} implies that $G$ also contains a transposition. Thus $G$ satisfies the hypotheses of Proposition \ref{prop:gen_sets} and we conclude that $G \simeq S_n$. \bigskip

\textit{Case 2: $n$ is odd.} Now we take $d_g = (n-1)/2$ and $d_h = (n-d)/2$. Fix a prime $p$ not dividing any nonzero coefficient of $f$. Consider an integral specialization $\a_0 = (a_0, \ldots, a_{(n-1)/2})$ and $\b_0 = (b_0, \ldots b_{(n-d)/2})$ with the following $p$-adic valuations:
\begin{align}\label{eq:odd_d_odd_n_n_cycle}
	v_p(a_0) &= 0\\
	\nonumber v_p(a_i) & \geq 2 \text{ for } i > 0\\
	\nonumber v_p(b_j) & \geq 2 \text{ for } j < (n-d)/2\\
	\nonumber v_p(b_{(n-d)/2}) & =1.
\end{align}
These requirements ensure that the constant term $a_0^2 - b_0^2c_0$ has valuation exactly 0, the leading coefficient $b_{(n-d)/2}^2c_d$ has valuation exactly 2, and all intermediate coefficients have valuation at least 2. This produces the $p$-adic Newton polygon for $F_{\a_0, \b_0}$ shown below in Figure \ref{fig:odd_d_odd_n_n_cycle}.

\begin{figure}[ht]
\centering
\caption{$\NP_{\mathbb{Q}_p} (F_{\a_0, \b_0})$ with $n$-cycle}
\label{fig:odd_d_odd_n_n_cycle}
\vspace{1ex}

\begin{tikzpicture}[scale=0.75]
	\draw[->, thick] (-0.2, 0) -- (10.2, 0);
	\draw[->, thick] (0, -0.2) -- (0, 2.2);
	
	\filldraw[black] (0, 0) circle (2pt) node[left] {$(0, 0)$};
	\filldraw[black] (10, 2) circle (2pt) node[right] {$(n, 2)$};
	
	\draw[black] (0,0) -- (10,2);
\end{tikzpicture}

\end{figure}

This Newton polygon has one segment of slope $2/n$, and since $n$ is odd we have $\gcd(2,n) = 1$. Thus Lemma \ref{lem:NP_factorization} implies that the specialization $F_{\a_0, \b_0}$ is irreducible over $\mathbb{Q}_p$, hence over $\mathbb{Q}$, and we have that $F$ must be irreducible over $\mathbb{Q}(\a, \b)$, with its Galois group $G$ a transitive subgroup of $S_n$.

Next, we aim to produce a $q$-cycle in $G$ for a prime $q > n/2$. We will assume $n > 3$ for now, as the case of $n = d = 3$ will be handled by later arguments. Recalling Bertrand's postulate, there exists some prime $q$ such that $\frac{n-1}{2} < q < n-1$, which is odd and satisfies $q > n/2$. Consider now a specialization $\a_0, \b_0$ satisfying
\begin{align}
\label{eq:odd_d_odd_n_q_cycle}
	v_p(a_{(n-q)/2}) &= 0 \\
	\nonumber v_p(a_i) & \geq 2 \text{ for } i \neq (n-q)/2\\
	\nonumber v_p(b_j) & \geq 2 \text{ for } j < (n-d)/2\\
	\nonumber v_p(b_{(n-d)/2}) &= 1.
\end{align}
These requirements ensure that the valuations of all coefficients of $F_{\a_0, \b_0}$ are at least 2, except for the degree $n-q$ term, whose coefficient has valuation zero coming from the presence of an $a_{(n-q)/2}^2$ term. The leading coefficient $b_{(n-d)/2}^2c_d$ has valuation exactly 2. An example  $p$-adic Newton polygon for such a specialization $F_{\a_0, \b_0}$ is shown below in Figure \ref{fig:odd_d_odd_n_q_cycle}.

\begin{figure}[ht]
\centering
\caption{$\NP_{\mathbb{Q}_p}(F_{\a_0,\b_0})$ with $q$-cycle}
\label{fig:odd_d_odd_n_q_cycle}
\vspace{1ex}

\begin{tikzpicture}[scale=0.75]
	\draw[->, thick] (-0.2, 0) -- (10.2, 0);
	\draw[->, thick] (0, -0.2) -- (0, 3.5);
	
	\filldraw[black] (0, 3) circle (2pt) node[left] {$(0, \geq 4)$};
	\filldraw[black] (3,0) circle (2pt) node[below] {$(n-q, 0)$};
	\filldraw[black] (10, 1.5) circle (2pt) node[above] {$(n, 2)$};
	
	\draw[black] (0,3) -- (3, 0);
	\draw[black] (3, 0) -- (10, 1.5);
\end{tikzpicture}

\end{figure}

Note that the left side of the Newton polygon in Figure \ref{fig:odd_d_odd_n_q_cycle} need not be a single segment, or if $n=q$ it will not exist at all. This is inconsequential however, because the right side is of interest to us, in particular the segment of slope $2/q$ and length $q$. Since $q > n/2$ is an odd prime, we have $\gcd(2,q) = 1$ and $q$ is coprime to any integers less than or equal to $n-q$, so Lemma \ref{lem:NP_cycle} applies, ensuring the existence of a $q$-cycle in $G_{\a_0, \b_0}$. Since a positive proportion of integral specializations satisfy \eqref{eq:odd_d_odd_n_q_cycle}, Corollary \ref{cor:hilbert_cor} implies that $G$ contains a $q$-cycle as well.

	Finally, we can produce a transposition in $G$ using essentially the same argument as in the case of even $n$. After a possible change of variables, let $p>n$ be a prime such that $v_p(c_0) = 1$ and $p \nmid \Disc f, c_d$. We consider specializations with the following $p$-adic valuations.
\begin{align}
\label{eq:odd_d_odd_n_transp}
	v_p(a_0) &= 2\\
\nonumber	v_p(a_1) &= 0\\
\nonumber	v_p(a_i) &\geq 2 \text{ for } 1 < i \leq (n-1)/2\\
\nonumber	v_p(b_0) = v_p(b_{(n-d)/2}) &= 1\\
\nonumber	v_p(b_j) &\geq 1 \text{ for } 0 < j < (n-d)/2.
\end{align}
These conditions produce the Newton polygon shown below in Figure \ref{fig:odd_d_odd_n_transp}. 

\begin{figure}[ht]
\centering
\caption{$\NP_{\mathbb{Q}_p}(F_{\a_0,\b_0})$ with transposition}
\label{fig:odd_d_odd_n_transp}
\vspace{1ex}

\begin{tikzpicture}[scale=0.75]
\draw[->, thick] (-0.2, 0) -- (10.5, 0);
	\draw[->, thick] (0, -0.2) -- (0, 3.5);
	
	\filldraw[black] (0, 3) circle (2pt) node[left] {$(0, 3)$};
	\filldraw[black] (2, 0) circle (2pt) node[below] {$(2, 0)$};
	\filldraw[black] (9, 2) circle (2pt) node[right] {$(n, 2)$};
	
	\draw[black] (0,3) -- (2,0);
	\draw[black] (2,0) -- (9,2);
\end{tikzpicture}
\end{figure}

Since $n-2$ is odd, Lemma \ref{lem:NP_cycle} applied to the segment of slope $-3/2$ implies that $G_{\a_0, \b_0}$, and hence $G$ by Corollary \ref{cor:hilbert_cor}, contains a transposition. Therefore, $G$ satisfies the hypotheses of Proposition \ref{prop:gen_sets}, and we conclude $G \simeq S_n$.
\end{proof}

\subsection{Even degree curves} We now present the analogous proposition for the case of $d$ even. Let $f(x) \in \mathbb{Z}[x]$ be squarefree given by $f(x) = \sum_{i=0}^d c_ix^i$, with $d \geq 4$ even. Fix an even integer $n \geq d+2$ and take $d_g = n/2$ and $d_h = (n-d)/2-1$. Let $F_f(\a,\b,x) \in \mathbb{Q}(\a,\b)[x]$ denote the polynomial family in \eqref{eq:F_f}, which is seen to have degree $n$. Again, for simplicity we denote this by $F(x)$ when it will not create confusion.

\begin{proposition}
\label{prop:even_irred}
	Fix a polynomial $f$, an even integer $n$, and degrees $d_g, d_h$ as above. Then $F_f$ is irreducible in $\mathbb{Q}(\a,\b)[x]$ and $\Gal(F_f/\mathbb{Q}(\a,\b)) \simeq S_n$.
\end{proposition}

\begin{proof}
 
	We will again need that the irreducibility of $F_f$ and its Galois group $G$ are invariant under linear change of coordinates in $x$, to allow us to assume certain conditions on the valuations of the $c_i$. 
	
		As in the proof of Proposition \ref{prop:odd_irred_Sn}, there exists a prime $p>n$ not dividing both $\Disc f, c_d$ such that $p$ divides $f(k)$ exactly once for some integer $k$. Thus after changing variables, we assume that $v_p(c_0) = 1$.
	
	Consider now the change of variables by scaling $x$ to be $px$. The constant term $c_0$ remains unchanged, but this allows us to assume that $p \mid c_i$ for $i \geq 1$. These assumptions are useful for finding long cycles in $G = \Gal(F_f/\mathbb{Q}(\a,\b))$. We consider an integral specialization $\a_0, \b_0$ with the following $p$-adic valuations:
	\begin{align}\label{eq:even_d_n_cycle}
		v_p(a_i) &\geq 1 \text{ for } i < n/2\\
		\nonumber v_p(a_{n/2}) &= 0 \\
		\nonumber v_p(b_0) &= 0,
	\end{align}
	and no restrictions on $b_j$ for $j > 0$. These restrictions, and assumptions on the coefficients $c_i$, ensure that every term of $F(x)$ is divisible by $p$, except for the leading coefficient $a_{n/2}^2$, which has valuation 0. Moreover, the valuation of the constant term $a_0^2 - b_0^2c_0$ is exactly 1, so the Newton polygon of $F_{\a_0, \b_0}$ has exactly one segment of length $n$ and slope $-1/n$, as shown in Figure \ref{fig:even_d_n_cycle}. 

\begin{figure}[ht]
\centering
\caption{$\NP_{\mathbb{Q}_p} (F_{\a_0, \b_0})$ with $n$-cycle}
\label{fig:even_d_n_cycle}
\vspace{1ex}

\begin{tikzpicture}[scale=0.75]
	\draw[->, thick] (-0.2, 0) -- (10.5, 0);
	\draw[->, thick] (0, -0.2) -- (0, 1.5);
	
	\filldraw[black] (0, 1) circle (2pt) node[left] {$(0, 1)$};
	\filldraw[black] (10, 0) circle (2pt) node[above] {$(n, 0)$};
	
	\draw[black] (0,1) -- (10,0);
\end{tikzpicture}
\end{figure}	
	
	Lemma \ref{lem:NP_cycle} implies that $F_{\a_0, \b_0}$ is irreducible over $\mathbb{Q}_p$, and hence over $\mathbb{Q}$, so $F$ is irreducible over $\mathbb{Q}(\a,\b)$ and $G$ is transitive, containing an $n$-cycle by Corollary \ref{cor:hilbert_cor}.
	
	We use a variation of this argument to find an $(n-1)$-cycle in $G$. Fix another prime $p>n$ such that after a change of variables we have $v_p(c_0) = 1$ and $p \nmid c_1$. We consider an integral specialization $\a_0, \b_0$ with the following $p$-adic valuations:
	\begin{align}\label{eq:even_d_n-1_cycle}
		v_p(a_i) &\geq 3 \text{ for } i < n/2\\
		\nonumber v_p(a_{n/2}) &= 0 \\
		\nonumber v_p(b_0) &= 1\\
		\nonumber v_p(b_j) & \geq 2 \text{ for } j > 0.
	\end{align}
	These restrictions ensure that the constant term has valuation 3, while the linear coefficient, $2a_0a_1 - b_0^2c_1 - 2b_0b_1c_0$, has valuation exactly 2. All other terms have valuation at least 2 except for the leading term, which has valuation 0. This produces the Newton polygon below in Figure \ref{fig:even_d_n-1_cycle}.
	
	\begin{figure}[ht]
\centering
\caption{$\NP_{\mathbb{Q}_p} (F_{\a_0, \b_0})$ with $(n-1)$-cycle}
\label{fig:even_d_n-1_cycle}
\vspace{1ex}

\begin{tikzpicture}[scale=0.75]
	\draw[->, thick] (-0.2, 0) -- (10.5, 0);
	\draw[->, thick] (0, -0.2) -- (0, 3.5);
	
	\filldraw[black] (0, 3) circle (2pt) node[left] {$(0, 3)$};
	\filldraw[black] (1, 2) circle (2pt) node[below] {$(1, 2)$};
	\filldraw[black] (10, 0) circle (2pt) node[above] {$(n, 0)$};
	
	\draw[black] (0,3) -- (1,2);
	\draw[black] (1,2) -- (10,0);
\end{tikzpicture}
\end{figure}	

	Since $n \geq 4$, the two segments are distinct, with the rightmost one of length $n-1$ and slope $-2/(n-1)$. As $n$ is even, Lemma \ref{lem:NP_cycle} is satisfied, producing an $(n-1)$-cycle in $G_{\a_0, \b_0}$ and thus in $G$.
	
	Finally, we produce a transposition in $G$, assuming that $n \geq 8$ for simplicity; nearly identical arguments suffice for the case of $d=4$ and $n=6$. More care is needed here to find a Newton polygon with exactly one segment of even length to satisfy the hypotheses of Lemma \ref{lem:NP_cycle}.
	
	Fix a prime $p > n$ such that $p \nmid c_d, \Disc f$, $c_d$ is a quadratic residue modulo $p$, and $p \mid f(k)$ for some integer $k$. Such a prime exists by our earlier Chebotarev argument, this time looking for primes splitting completely in the splitting field of $f(x)(x^2 - c_d)$.	After a change of coordinates, we assume $v_p(c_0) = 1$ and $p \nmid c_1$. We consider an integral specialization $\a_0, \b_0$ with the following restrictions:
	\begin{align}\label{eq:even_d_transp}
		v_p(a_i) &\geq 4 \text{ for } i < \frac{n}{2}-2\\
	\nonumber v_p(a_{n/2-2}) &= 0\\
	\nonumber v_p(a_{n/2-1}) &= 1\\
	\nonumber v_p(a_{n/2}) &= 1 \text{ such that } \frac{a_{n/2}^2}{p^2} \equiv c_d \pmod{p^2}\\
	\nonumber v_p(b_0) &= 1\\
	\nonumber v_p(b_j) &\geq 1\\
	\nonumber v_p(b_{(n-d)/2}) &= 1 \text{ such that } \frac{b_{(n-d)/2}^2}{p^2} \equiv 1 \pmod{p^2}.
	\end{align}
	Note that such $a_{n/2}$ exists, since $c_d$ is a quadratic residue, and these assumptions ensure that $p^4 \mid a_{n/2}^2 - b_{(n-d)/2}^2c_d$, the leading coefficient. Furthermore, we have that the constant coefficient has valuation 3, the linear coefficient has valuation 2, the $x^{n-4}$ coefficient has valuation 0, and both the $x^{n-3}$ and  $x^{n-2}$ coefficients have valuation 1, with all other terms having valuation at least 2. 
	
	Looking more closely at the coefficient of $x^{n-1}$ given by 
	\[2a_{n/2-1}a_{n/2} - b_{(n-d)/2-1}b_{(n-d)/2}c_d - b_{(n-d)/2}^2c_{d-1},\]
	we see that its valuation at least 2. To ensure it has valuation exactly 2, we fix a residue class for $\frac1pb_{(n-d)/2-1}$ modulo $p$ and ask that $a_{n/2-1}$ satisfy
	\begin{equation}\label{eq:condition_on_a_n/2-1}
		\frac{a_{n/2-1}}{p} \not\equiv \left(2\frac{a_{n/2}}{p}\right)^{-1}\frac{1}{p^2}\left(b_{(n-d)/2-1}b_{(n-d)/2}c_d - b_{(n-d)/2}^2c_{d-1}\right) \pmod{p}.
	\end{equation}
	Thus combining \eqref{eq:even_d_transp} and \eqref{eq:condition_on_a_n/2-1}, we produce the Newton polygon in Figure \ref{fig:even_d_transp} below.	
	
\begin{figure}[ht]
\centering
\caption{$\NP_{\mathbb{Q}_p} (F_{\a_0, \b_0})$ with transposition}
\label{fig:even_d_transp}
\vspace{1ex}	
	\begin{tikzpicture}[scale=0.75]
	\draw[->, thick] (-0.2, 0) -- (10.5, 0);
	\draw[->, thick] (0, -0.2) -- (0, 4.5);
	
	\filldraw[black] (0, 3) circle (2pt) node[left] {$(0, 3)$};
	\filldraw[black] (1, 2) circle (2pt) node[below] {$(1, 2)$};
	\filldraw[black] (6, 0) circle (2pt) node[below] {$(n-4,0)$};
	\filldraw[black] (8,1) circle (2pt) node[right] {$(n-2,1)$};
	\filldraw[black] (9,2) circle (2pt) node[right] {$(n-1,2)$};
	\filldraw[black] (10, 4) circle (2pt) node[above] {$(n, \geq 4)$};
	
	\draw[black] (0,3) -- (1,2);
	\draw[black] (6,0) -- (1,2);
	\draw[black] (6,0) -- (8,1);
	\draw[black] (9,2) -- (8,1);
	\draw[black] (9,2) -- (10,4);
\end{tikzpicture}
\end{figure}

	The segment of length 2 and slope 1/2, together with the fact that all other segments have odd length $l'$ and slopes $r'/l'$ with $\gcd(r',l') = 1$, allow us to apply Lemma \ref{lem:NP_cycle} with $l=2$ to produce a transposition in $G_{\a_0, \b_0}$. The requirements \eqref{eq:even_d_transp} and \eqref{eq:condition_on_a_n/2-1} are satisfied for a positive proportion of integral $\a_0, \b_0$, so Corollary \ref{cor:hilbert_cor} implies that $G$ contains a transposition. Thus with its $n$-cycle, $(n-1)$-cycle, and transposition, Proposition \ref{prop:gen_sets} gives that $G \simeq S_n$.	
	\end{proof}

\section{Relating coefficients to roots}
\label{sec:intermediates}

In this brief section we state a result which relates the absolute value of a polynomial's coefficients to that of its roots, which will be useful later when counting multiplicities of fields generated by a family of polynomials. To avoid confusion, we note that for the purposes of this section $f(x)$ denotes a general polynomial in $\mathbb{C}[x]$, rather than squarefree integral polynomial defining a nonsingular hyperelliptic curve, as in the previous section.

\begin{lemma}
\label{lem:roots_coeff_bound}
	Let $f(x) = \sum_{i=0}^n c_ix^{i} \in \mathbb{C}[x]$ be monic and have degree $n$. There exist positive constants $k_i$ such that for any $Y > 0$, if $|c_i| \leq k_iY^{n-i}$ for $0 \leq i \leq n$ then $|\alpha| \leq Y$ for all roots $\alpha$ of $f(x)$.
\end{lemma}

\begin{proof}
	This result follows from classical upper bounds on the absolute value of complex roots given by Lagrange and Cauchy. A clean proof yielding explicit values of $k_i$ follows from the following bound due to Fujiwara \cite{fujiwara},
	\begin{equation}\label{eq:fujiwara}
		|\alpha| \leq 2 \max \left\lbrace \left|\frac{c_{n-1}}{c_n}\right|, \left|\frac{c_{n-2}}{c_n}\right|^{1/2}, \ldots, \left|\frac{c_{1}}{c_n}\right|^{1/(n-1)}, \left|\frac{c_{0}}{2c_n}\right|^{1/n} \right\rbrace.
	\end{equation}
	
	Set $k_0 = \frac{1}{2^{n-1}}$ and $k_i = \frac{1}{2^{n-i}}$ for $1 \leq i \leq n$. In our case we have $c_n = 1$, so if $|c_0| \leq \frac{1}{2^{n-1}}Y^n$, then $|\frac{c_0}{2}| \leq \frac{1}{2^n}Y^n$. Taking $n$-th roots, we have $|\frac{c_0}{2}|^{1/n} \leq \frac{Y}{2}$. Similarly, for $1 \leq i \leq n-1$, we have $|c_i| \leq \left(\frac{Y}{2}\right)^{n-i}$, so taking $(n-i)$-th roots implies $|c_i|^{1/(n-i)} \leq \frac{Y}{2}$. Thus
	\[\max \left\lbrace |c_{n-1}|, |c_{n-2}|^{1/2}, \ldots, |c_{1}|^{1/(n-1)}, \left|\frac{c_{0}}{2}\right|^{1/n} \right\rbrace \leq \frac{Y}{2},\] so applying \eqref{eq:fujiwara} gives $|\alpha| \leq Y$ for any root $\alpha$ of $f(x)$.
\end{proof}

\section{Proofs of main theorems}
\label{sec:odd}

We begin with the proof of the first bound in Theorem \ref{thm:odddeg}, which covers Sections \ref{sec:param} - \ref{sec:assembly}. In Section \ref{sec:improve} we describe the modifications necessary to obtain the improved bound in Theorem \ref{thm:odddeg} for sufficiently large $n$. The proof of Theorem \ref{thm:evendeg} is nearly identical, and we highlight the differences in Section \ref{sec:even}.

\subsection{Parameterization}\label{sec:param}

Let $C$ be a nonsingular hyperelliptic curve over $\mathbb{Q}$ of odd degree $d = 2g + 1$. Then $C$ has a model
\[C \colon y^2 = f(x) = \sum_{i=0}^d c_ix^i\]
where $c_i \in \mathbb{Z}$ for all $i$ and $f(x)$ is squarefree. We may further assume that $c_0 \neq 0$ by translating $x$ if needed. If necessary, we may also take $f$ to be monic, by multiplying by $c_d^{d-1}$ and changing variables again.

Let $Y$ be a positive real number and $n \geq d$ an integer. We now construct a family of polynomials $P_{f,n}(Y)$ arising from certain specializations of \eqref{eq:F_f}. When $n$ is even, take
	\begin{align}\label{eq:d_odd_n_even_family}
		g(x) &= x^{n/2} + a_{n/2-1}x^{n/2-1} + ... +  a_0\\
	\nonumber	h(x) &= b_{(n-d-1)/2}x^{(n-d-1)/2} + b_{(n-d-1)/2 - 1}x^{(n-d-1)/2 -1} + ... + b_0,
	\end{align}
with the restrictions that $a_i, b_j \in \mathbb{Z}$, $|a_{n/2-i}| \leq Y^i$, and $|b_{(n-d-1)/2 - j}| \leq Y^{j+1/2}$. In the case of $n$ odd we take
	\begin{align}\label{eq:d_odd_n_odd_family}
		g(x) &= a_{(n-1)/2}x^{(n-1)/2} + a_{(n-1)/2 -1 }x^{(n-1)/2-1} + ... +  a_0\\
	\nonumber	h(x) &= x^{(n-d)/2} + b_{(n-d)/2 - 1}x^{(n-d)/2 -1} + ... + b_0,
	\end{align}
with $|a_{(n-1)/2-i}| \leq Y^{i+1/2}$ and $|b_{(n-d)/2 - j}| \leq Y^j$. 

Let $P_{f,n}(Y)$ be the set of polynomials $F(x) = g(x)^2 - f(x)h(x)^2$ for $g(x),h(x)$ of the form above. Note that any such $F(x)$ has degree $n$, and by Lemma \ref{lem:roots_coeff_bound} any root $\alpha$ of $F$ satisfies $|\alpha| \ll_{n,f} Y$. Hence $\Disc(F) \leq k Y^{n(n-1)}$ for a constant $k$ depending on $f$ and $n$.

\subsection{Bounding multiplicities}\label{sec:mult}

We can count the number fields arising from specializations of \eqref{eq:F_f} by counting elements of $P_{f,n}(Y)$, provided that we can control the multiplicity. This multiplicity arises from two sources. We may have different choices of $g(x)$ and $h(x)$ that produce the same element $F(x) \in P_{f,n}(Y)$, or we may find multiple elements of $P_{f,n}(Y)$ that produce isomorphic number fields. We deal with the former case in the following lemma.

\begin{lemma}
\label{lem:poly_mult}
	Let $F(x) \in \mathbb{Z}[x]$ be a polynomial of degree $n$. The number of ways to choose $g(x), h(x) \in \mathbb{Z}[x]$ with at least one monic so that $F(x) = g(x)^2 - f(x)h(x)^2$ is $O_n(1)$.
\end{lemma}

\begin{proof}
	Note that $f$ has no repeated roots, so the complex affine coordinate ring, given by $\mathbb{C}[x,y]/(y^2-f(x))$, is a Dedekind domain.  With this, one follows the justification of \cite[Lemma 7.4]{LOT} to argue that $(F)$ factors uniquely into a product of $2n$ prime ideals.
	
	Given any such $g,h$ with $F = g^2 - fh^2$, we associate the factorization $F = (g-\sqrt{f}h)(g + \sqrt{f}h)$. Thus the ideal $(g-\sqrt{f}h)$ factors as a product of the $2n$ primes dividing $(F)$, giving at most $2^{2n}$ possibilities for the ideal $(g-\sqrt{f}h)$. 
	
	Since $\deg f$ is odd, the units in the coordinate ring consist of the constants, as an element $u + \sqrt{f}v$ has norm $u^2 - fv^2$, which is a unit in $\mathbb{C}[x]$ if and only if $v=0$ and $u$ is a nonzero constant. Thus the ideal $(g-\sqrt{f}h)$ determines $g$ and $h$ exactly by the monicity assumption.
\end{proof}

Now we can give a count for $\#P_{f,n}(Y)$, since Lemma \ref{lem:poly_mult} gives that each choice of $a_i$ and $b_j$ above coincides with at most a constant number of other choices. In the case of $n$ even, we have $\#P_{f,n}(Y) \asymp Y^c$ where
\begin{align}\label{eq:count_polys_before_mult}
	\nonumber c = \sum_{i = 1}^{n/2} i + \sum_{j = 0}^{(n-d-1)/2} (j + 1/2) &= \frac{1}{4} \left( n^2 + (2-d)n + \frac{d^2 - 2d + 1}{2} \right)\\
	& = \frac{1}{4} \left(n^2 + (1-2g)n + 2g^2 \right).
\end{align} 
The same approach yields the same count for $n$ odd. Since the elements of $P_{f,n}(Y)$ arise as specializations of the family \eqref{eq:F_f}, Proposition \ref{prop:odd_irred_Sn} implies that $\#P_{f,n}(Y, S_n) \asymp Y^c$ where $c$ is given in \eqref{eq:count_polys_before_mult} and $P_{f,n}(Y,S_n)$ is the subset consisting of irreducible $F \in P_{f,n}(Y)$ with $\Gal(F/\mathbb{Q}) \simeq S_n$.

We now address the second source of multiplicity, namely that there may be many $F \in P_{f,n}(Y)$ for which $K \simeq \mathbb{Q}[x]/F(x)$. To deal with this, we employ machinery developed by Ellenberg and Venkatesh \cite{ellenbergvenkatesh} for counting number fields, and the multiplicity counts of Lemke Oliver and Thorne \cite{LOT}.

Following their lead we define
\[S(Y) := \set{ F = x^n + c_{n-1}'x^{n-1} + ... + c_0' \in \mathbb{Z}[x] : \left|c_{n-i}' \right| \ll_{n,f} Y^i }\]
with the condition that  $F(x)$ is irreducible. Note that by this construction $P_{f,n}(Y, S_n) \subseteq S(Y, S_n)$, provided we choose the implied constant appropriately. We now define for a number field $K$ its multiplicity within $S(Y, S_n)$,
\[M_K(Y) := \# \set{F \in S(Y) \mid \mathbb{Q}[x]/F(x) \simeq K}.\]

\begin{lemma}[Lemke Oliver -- Thorne, {\cite[Proposition 7.5]{LOT}}]
\label{lem:LOT_mult}
	We have \[M_K(Y) \ll \max\Big(Y^n\left|\Disc(K)\right|^{-1/2}, Y^{n/2}\Big).\]
\end{lemma}

The proof uses the geometry of numbers, building on the strategy suggested in \cite{ellenbergvenkatesh}.

We now state an upper bound for the asymptotics of general number field counts, without respect to any curve. We use $N_n(X)$ to denote the number of degree $n$ number fields $K$ with $\left|\Disc(K)\right| \leq X$.
\begin{theorem}[Schmidt, \cite{schmidt}]
\label{prop:counting_fields_general}
	For $n \geq 3$, we have
	\begin{equation}\label{eq:schmidt_bound}
		N_n(X) \ll X^{\frac{n + 2}{4}}.
	\end{equation}
\end{theorem}

We leverage Theorem \ref{prop:counting_fields_general} to show that the contribution to $N_{n,C}(X, S_n)$ by fields of low discriminant is negligible. This allows for some improvement to the eventual exponent.

\begin{lemma}
\label{lem:sum_small_disc}
	Let $T \leq Y^n$. Then \[\sum_{\left|\Disc(K)\right| \leq T} M_K(Y) \ll Y^nT^{n/4},\] where the sum runs over all degree $n$ number fields $K$ with $\left|\Disc(K)\right| \leq T$.
\end{lemma}

\begin{proof}
	We begin by rewriting the sum as a Riemann-Stieljes integral,
	\begin{align}
		\nonumber \sum_{\left|\Disc(K)\right| \leq T} M_K(Y) &= \sum_{1 \leq t \leq T} \left(N_n(t) - N_n(t-1)\right)M_K(Y)(t) \\
		\label{eq:mult_int} &= \int_{1^-}^T M_K(Y)(t)dN_n(t)\\
		\label{eq:mult_int2} &\ll Y^n\int_{1^-}^T \frac{dN_n(t)}{t^{1/2}},
	\end{align}
	where \eqref{eq:mult_int2} follows from \eqref{eq:mult_int} by the multiplicity bound from Lemma \ref{lem:LOT_mult}. Integrating by parts in \eqref{eq:mult_int2} produces
	\begin{align}\label{eq:ibp}
		Y^n\int_{1^-}^T \frac{dN_n(t)}{t^{1/2}} &= Y^n\frac{N_n(T)}{T^{1/2}} + \frac{Y^n}{2} \int_{1-}^T \frac{N_n(t)}{t^{3/2}} dt.
	\end{align}
	
	Recalling Schmidt's bound in \eqref{eq:schmidt_bound}, we estimate \eqref{eq:ibp} by
	\begin{align*}%\label{eq:ibp_schmidt}
		\nonumber Y^n\frac{N_n(T)}{T^{1/2}} + \frac{Y^n}{2} \int_{1^-}^T \frac{N_n(t)}{t^{3/2}}dt & \ll Y^nT^{n/4} + \frac{Y^n}{2} \int_1^T t^{\frac{n}{4} - 1} dt \\
		\nonumber & = Y^nT^{n/4} + \frac{2Y^n}{n}(T^{n/4} - 1) \\
		\nonumber & = Y^n\left(\left(1 + \frac{2}{n}\right)T^{n/4} - \frac{2}{n}\right) \\
		& \ll Y^nT^{n/4}.	\qedhere \end{align*}\end{proof}

\subsection{Final steps}\label{sec:assembly}

We are now ready to assemble the proof of Theorem \ref{thm:odddeg}. By our construction, for any $F \in P_{f,n}(Y, S_n)$ and any root $\alpha$ of $F$, we have $\left(\alpha, \frac{g(\alpha)}{h(\alpha)}\right) \in C(K)$ where $K = \mathbb{Q}(\alpha)$ is a field of degree $n$ with $\Gal(\widetilde{K}/\mathbb{Q}) \simeq S_n$. We then have that $\left|\Disc(K)\right| \leq kY^{n(n-1)}$ for a constant $k$ depending on $f,n$.

Choose $T = \kappa Y^{n - (3+2g) + 2g^2/n}$ for a positive constant $\kappa$ to be determined shortly. By Lemma \ref{lem:sum_small_disc}, we have 
	\begin{equation}
	\label{eq:small_disc}
		\sum_{\left|\Disc(K)\right| \leq T} M_K(Y) \ll \kappa^{n/4}Y^c,
	\end{equation} 
	and we recall from our earlier discussion that
	\begin{equation}
	\label{eq:poly_asymp}
		\#P_{f,n}(Y, S_n) \asymp Y^c.
	\end{equation} 
	We then choose $\kappa$ sufficiently small so that the quantity in \eqref{eq:small_disc} is at most $\#P_{f,n}(Y,S_n)/2$. Then, fields $K$ with $T < \left| \Disc(K) \right| \leq kY^{n(n-1)}$ arise from a positive proportion of the polynomials in $P_{f,n}(Y, S_n)$. Counting just these fields and recognizing the bound for $M_K(Y)$ in Lemma \ref{lem:LOT_mult} is decreasing with respect to $\left|\Disc(K)\right|$, we have $M_K(Y) \ll T^{-1/2}Y^n$ for all $K$ with $T <\left|\Disc(K)\right| \leq kY^{n(n-1)}$. Thus we have 
\begin{align}\label{eq:almost_done}
	\nonumber N_{n,C}(kY^{n(n-1)}, S_n) &\gg Y^{c-n}T^{1/2}\\
	&=  Y^{\frac{1}{4}\left(n^2 - (1+2g)n + 2g^2 - 4g - 6 + 4g^2/n\right)}.
\end{align}
Upon replacing $Y$ in \eqref{eq:almost_done} by $\left(X/k\right)^{1/n(n-1)}$ and simplifying, we obtain as the exponent
\[c_n= \frac{1}{4} - \frac{gn^2 - (g^2 - 2g - 3)n - 2g^2}{2n^2(n-1)}\]
and thus $N_{n,C}(X,S_n) \gg X^{c_n}$, which is the first statement of Theorem \ref{thm:odddeg}.

\subsection{Improvements}\label{sec:improve}

To improve the exponent in the previous section, we seek to find when fields of discriminant less than $Y^n$ contribute negligibly, allowing us to use the best possible multiplicity bound in Lemma \ref{lem:LOT_mult}, $M_K(Y) \ll Y^{n/2}$. If we assume this is true for some $n$, then we immediately have
\[N_{n,C}(kY^{n(n-1)}, S_n) \gg Y^{c - n/2},\]
and after simplifying and making the same substitutions as earlier, we obtain the improvement in Theorem \ref{thm:odddeg}.

It now remains to argue that this is possible. Suppose that $N_n(X) \ll X^{\alpha(n, g)}$ is valid for large enough $n$. With this assumption, we use the same procedure as the proof of Lemma \ref{lem:sum_small_disc} to show that
\begin{equation}\label{eq:alpha_ng}
	\sum_{\left|\Disc K\right| \leq Y^n} M_K(Y) \ll Y^{n/2 + n\alpha(n,g)}
\end{equation}
To make the right hand side of \eqref{eq:alpha_ng} be $o(Y^c)$, it suffices to take any $\alpha(n,g)$ satisfying
\begin{equation}\label{eq:alpha_ineq}
	\alpha(n,g) < \frac{n}{4} - \frac{1 + 2g}{4} + \frac{g^2}{2n}.
\end{equation}

Theorem \ref{prop:counting_fields_general} is insufficient for this purpose. We turn to the improved upper bounds for counting number fields of fixed degree by discriminant due to Lemke Oliver and Thorne \cite{LOT_upper_bounds}.

\begin{theorem}[Lemke Oliver -- Thorne, {\cite[Theorems 1.1, 1.2]{LOT_upper_bounds}}]
\label{thm:LOT_upper}
	For $n \geq 6$ we have
	\begin{equation}\label{eq:LOT_1.1}
		N_n(X) \ll X^{1.564 (\log n)^2}.
	\end{equation}
	Moreover, for $n \geq 2$ we have the following.
	\begin{enumerate}
		\item Let $m$ be the least integer for which $\binom{m+2}{2} \geq 2n + 1$. Then $N_n(X) \ll X^{2m - \frac{m(m-1)(m+4)}{6n}}$.
		
		\item Let $3 \leq r \leq n$ and let $m$ be an integer such that $\binom{m+r-1}{r-1} > rn$. Then $N_n(X) \ll X^{mr}.$
	\end{enumerate}
\end{theorem}\bigskip

By taking $\alpha(n,g) = 1.564 (\log n)^2$, as in Theorem \ref{thm:LOT_upper}, we see that \eqref{eq:alpha_ineq} is satisfied for $n$ sufficiently large, since $(\log n)^2$ grows more slowly than the right hand side for any fixed $g$. This completes our justification of the improved exponent in Theorem \ref{thm:odddeg}. 

For any fixed $g$, one can compute the $n$ at which \eqref{eq:alpha_ineq} takes effect for $\alpha(n,g) = 1.564(\log n)^2$. Then, one can use (1) and (2) of Theorem \ref{thm:LOT_upper} to search by computer for the least $n$ for which there exists $\alpha(n,g)$ such that $N_n(X) \ll X^{\alpha(n,g)}$ and \eqref{eq:alpha_ineq} is satisfied, by checking all appropriate pairs of integers $(m,r)$. When $g=1$, the improved exponent is valid for all $n \geq 106$. When $g = 10$, this approach shows the improved exponent is valid for $n \geq 138$. For $g = 100$, this increases to $n \geq 324$. 

Since Theorem \ref{thm:odddeg} is only valid for degrees $n \geq d = 2g+1$, when $g$ is sufficiently large, the improved exponent will be valid for all $n \geq d$. We computed this to be true for all $g \geq 238$.

\subsection{Even degree curves}\label{sec:even}

The proof of Theorem \ref{thm:evendeg} follows the approach of the previous subsection. We begin with a hyperelliptic curve $C \colon y^2 = f(x)$ with $f(x) = \sum_{i = 0}^d c_ix^i \in \mathbb{Z}[x]$ squarefree for $d \geq 4$ even. 

For $Y > 0$ and an even integer $n \geq d + 2$ we define a family $P_{f,n}(Y)$ by polynomials of the form $F(x) = g(x)^2 - f(x)h(x)^2$, where
	\begin{align*}
		g(x) &= x^{n/2} + a_{n/2-1}x^{n/2-1} + ... +  a_0\\
		h(x) &= b_{(n-d)/2 - 1}x^{(n-d)/2-1} + b_{(n-d)/2 - 2}x^{(n-d)/2 -2} + ... + b_0,
	\end{align*}
	satisfy $|a_{n/2 -i}| \leq Y^{i}$ and $|b_{(n-d)/2 - j}| \leq Y^j$. We address the possibility that multiple choices of $g(x)$ and $h(x)$ produce coinciding $F(x)$ in the following analogue to Lemma \ref{lem:poly_mult}.
	
\begin{lemma}\label{lem:poly_mult_even}
	Let $F(x) \in \mathbb{Z}[x]$ be a polynomial of even degree $n$. The number of ways to choose $g(x) \in \mathbb{Z}[x]$ monic of degree $n/2$ and $h(x) \in \mathbb{Z}[x]$ of degree $(n-d)/2-1$ such that $F(x) = g(x)^2 - f(x)h(x)^2$ is $O_n(1)$.
\end{lemma}	

\begin{proof}
	By the first paragraph of the proof of Lemma \ref{lem:poly_mult}, we have that the ideal $(F)$ in the ring $\mathbb{C}[x,y]/(y^2 = f(x))$ factors uniquely into a product of $2n$ primes. Again, any such $g,h$ give us a factorization $F = (g-\sqrt{f}h)(g+\sqrt{f}h)$ with at most $2^{2n}$ possibilities for the ideal $(g-\sqrt{f}h)$.
	
	Unlike in Lemma \ref{lem:poly_mult}, there may be nontrivial units in the coordinate ring. Suppose $F= g^2 - fh^2 = (g')^2 - f(h')^2$ and $(g-\sqrt{f}h) = (g' - \sqrt{f}h')$ for some $g',h'$ of degrees $n/2, (n-d)/2-1$ respectively. Thus for some unit in the coordinate ring of the form $u + \sqrt{f}v$, i.e.\ we have
	\begin{align*}
		gu-fhv &= g' \\
		hu - gv &= h'.
	\end{align*}
	The above implies that 
	\[g'h - gh' = h(gu - fhv) - g(hu-gv) = v(g^2 - fh^2) = vF.\]
	However, the degree of the left hand side is at most $\deg g + \deg h < n$, while the degree of the right hand side is at least $n$ if $v$ is nonzero, a contradiction. Therefore, only nonzero constants preserve both the ideal $(g - \sqrt{f}h)$ and the desired degrees of $g$ and $h$, so the monicity assumption ensures that the ideal determines $g$ and $h$ precisely.
\end{proof}
	
	By Lemma \ref{lem:poly_mult_even} and the same argument as for the odd degree case, we have $\#P_{f,n}(Y) \asymp Y^c$ for
	\begin{equation}\label{eq:even_Pf_count}
		c = \frac{1}{4} \left(n^2 - 2gn + 2g^2 + 2g \right).
	\end{equation}
Proposition \ref{prop:even_irred} guarantees that a positive proportion of the elements of $P_{f,n}(Y)$ are irreducible of degree $n$ and have Galois group $S_n$.

We define $S(Y)$ and $M_K(Y)$ as in the odd degree case. Taking $T = \kappa Y^{n - (4 + 2g) + \frac{2g^2 + 2g}{n}}$ and applying Lemma \ref{lem:sum_small_disc}, we obtain the analogue of \eqref{eq:small_disc},
\begin{equation}\label{eq:small_disc_even}
	\sum_{\left|\Disc(K)\right| \leq T} M_K(Y) \ll \kappa^{n/4} Y^c,
\end{equation}
with $c$ as in \eqref{eq:even_Pf_count}. As before, we choose $\kappa$ sufficiently small so the left hand side of \eqref{eq:small_disc_even} is at most $\#P_{f,n}(Y, S_n)/2$, allowing us to only count the contribution of fields $K$ with $T < \Disc(K) kY^{n(n-1)}$. Proceeding as in \eqref{eq:almost_done}, we have
\begin{equation*}\label{eq:almost_done_even}
	N_{n,C}(kY^{n(n-1)}, S_n) \gg Y^{c-n}T^{1/2} = Y^{\frac{1}{4} \left( n^2 - (2 + 2g)n + 2g^2 - 2g - 8 + \frac{4g^2 + 4g}{n} \right)}.
\end{equation*}
Replacing $Y$ above by $(X/k)^{1/n(n-1)}$ we obtain $N_{n,C}(X, S_n) \gg X^{c_n}$, with
\[c_n = \frac{1}{4} - \frac{(1 + 2g)n^2 - (2g^2 - 2g - 8)n - (4g^2 + 4g)}{4n^2(n-1)},\]
as in Theorem \ref{thm:evendeg}.

To obtain the improved lower bound when $n$ is sufficiently large, the procedure is identical to that of Section \ref{sec:improve}. The improved bound is again $N_{n,C}(kY^{n(n-1)}, S_n) \gg Y^{c - n/2}$, but with $c$ given by \eqref{eq:even_Pf_count} instead, leading to the exponent in the second statement of Theorem \ref{thm:evendeg}. For this to be valid, we need $N_{n}(X) \ll X^{\alpha(n,g)}$ with
\begin{equation}\label{eq:alpha_even}
	\alpha(n,g) < \frac{n}{4} - \frac{1 + g}{2} + \frac{g^2 + g}{2n}.
\end{equation}
This is satisfied for $n$ sufficiently large by \eqref{eq:LOT_1.1} of Theorem \ref{thm:LOT_upper}. A computer search using (1) and (2) of Theorem \ref{thm:LOT_upper} can be used to explicitly find when the improved exponent takes effect. These come out to be quite similar to the odd degree case; for example, when $g = 1$, the improved exponent is valid for all $n \geq 108$. When $g=10$, it is valid for all $n \geq 139$. For $g=100$, this increases to $n \geq 325$. As in the previous case, for $g$ sufficiently large, the improved exponent of Theorem \ref{thm:evendeg} will be valid for all degrees $n \geq d+2$. We computed this to be true for all $g \geq 237$.

\bibliographystyle{alpha}
\bibliography{Keyes_GrowthPointsHECs_bib}

\end{document}